\documentclass[english]{amsart}
\usepackage[T1]{fontenc}
\usepackage[latin9]{inputenc}
\usepackage{amsthm}
\usepackage{amsmath}
\usepackage{amssymb}
\usepackage{color}

\makeatletter
\numberwithin{figure}{section}
  \theoremstyle{plain}
  \newtheorem{thm}{\protect\theoremname}[section]
  \theoremstyle{plain}
  \newtheorem{lem}{\protect\lemmaname}[section]
 \ifx\proof\undefined\
   \newenvironment{proof}[1][\proofname]{\par
     \normalfont\topsep6\p@\@plus6\p@\relax
     \trivlist
     \itemindent\parindent
     \item[\hskip\labelsep
           \scshape
       #1]\ignorespaces
   }{%
     \endtrivlist\@endpefalse
   }
   \providecommand{\proofname}{Proof}
 \fi
  \theoremstyle{plain}
  \newtheorem{cor}{\protect\corollaryname}[section]
  \theoremstyle{plain}
  \newtheorem{prop}{\protect\propositionname}[section]
  \theoremstyle{remark}
  \newtheorem{rem}{\protect\remarkname}[section]
  \theoremstyle{definition}
  \newtheorem{defn}{\protect\definitionname}[section]
  \theoremstyle{plain}
  \newtheorem{fact}{\protect\factname}[section]

\makeatother

\usepackage{babel}
\providecommand{\corollaryname}{Corollary}
\providecommand{\definitionname}{Definition}
\providecommand{\factname}{Fact}
\providecommand{\lemmaname}{Lemma}
\providecommand{\propositionname}{Proposition}
\providecommand{\remarkname}{Remark}
\providecommand{\theoremname}{Theorem}

\begin{document}

\title{Some remarks on smooth renormings of Banach spaces}

\subjclass[2000]{46B03, 46B10.}

\author{Petr H\'ajek}
\address{Petr H\'ajek: Mathematical Institute\\Czech Academy of Science\\\v Zitn\'a 25 \\115 67 Praha 1\\
Czech Republic and Department of Mathematics\\Faculty of Electrical Engineering\\
Czech Technical University in Prague\\ Zikova 4, 160 00, Prague}
\email{hajek@math.cas.cz}

\author{Tommaso Russo}
\address{Tommaso Russo: Dipartimento di matematica\\ Universit\`a degli Studi di Milano\\
via Saldini 50, 20133 Milano, Italy}
\email{tommaso.russo@unimi.it}

\thanks{Research of the first author was supported in part by GA\v CR 16-07378S, RVO: 67985840.
Research of the second author was supported in part by the Universit\`a degli Studi di Milano (Italy) and in part by the Gruppo Nazionale
per l'Analisi Matematica, la Probabilit\`a e le loro Applicazioni (GNAMPA) of the Istituto Nazionale di Alta Matematica (INdAM) of Italy.}

\keywords{Fr\'{e}chet smooth, $C^{k}$-smooth norm, approximation of norms,
Minkowski functional, renorming, Implicit Function Theorem}

\subjclass[2010]{Primary 46B03; 46T20; Secondary 47J07; 14P20}

\date{\today}

\begin{abstract}
We prove that in every separable Banach space $X$ with a Schauder basis
and a $C^k$-smooth norm it is possible to approximate, uniformly on bounded sets,
every equivalent norm with a $C^k$-smooth one in a way that the approximation
is improving as fast as we wish on the elements depending only on the tail of the Schauder basis.

Our result solves a problem from the recent monograph of Guirao, Montesinos and Zizler.
\end{abstract}

\maketitle

\section{Introduction}

The problem of smooth approximation of continuous mappings  is one of the classical themes in analysis.
An important special case of this problem is the existence of
$C^k$-smooth  approximations of 
norms on an infinite-dimensional real Banach space. More precisely, assume that the real Banach space
$X$ admits a $C^k$-smooth norm. Let $\|\cdot\|$ be an equivalent norm
on $X$, $\varepsilon>0$. Does there exist a $C^k$-smooth renorming $\||\cdot\||$ of $X$ 
such that $1\le\frac{\||x\||}{\|x\|}\le1+\varepsilon$ holds for all $0\ne x\in X$?

In its full generality, this problem is still open, even in the case $k=1$ (no counterexample is known).
For $k=1$, the problem can be solved easily by using
Smulyan's criterion, once a dual LUR norm is present on $X^*$.
This covers a wide range of Banach spaces, in particular all WCG  (hence all separable, and all reflexive)
spaces \cite{dgz}. In the absence of a dual LUR renorming, the
problem appears to be completely open.

For $k\ge2$ the problem seems to be more difficult, and no dual approach
is available. 
To begin with, Deville \cite{deville-very-smooth} proved that the existence of $C^2$-smooth norm has profound structural
consequence for the space. In some sense, such spaces are either superreflexive, or
close to $c_0$. To get an idea of the difficulty of
constructing smooth norms, we refer to e.g. \cite{maalev-troyanski}, \cite{haydon-smooth}, 
\cite{haydon-trees}, \cite{haha}, \cite{bi}.

Broadly speaking, the construction of the smooth norm
is carried out by techniques locally using only finitely many ingredients.
Of course, this idea is present already in the 
concept of partitions of unity, but in the setting
of norms it is harder to implement as we need to preserve the
convexity of the involved functions.
Probably the first explicit use of this technique in order to construct smooth norms
is found in the work of Pechanec, Whitfield and Zizler \cite{pewhiziz}.
The authors construct
a particular LUR and $C^1$-smooth norm on $c_0(\Gamma)$ which admits $C^\infty$-approximations. 
This result has later been  
generalized to arbitrary WCG spaces \cite{hp}. Recently,
Bible and Smith \cite{bis} have succeeded  in solving the smooth
approximation problem for norms on $c_0(\Gamma)$, $k=\infty$. This is
essentially the only known nonseparable space where the
problem has been solved. 

In the separable setting, the problem has been completely
solved for every separable Banach space and every $k$, 
in a series of papers \cite{hajek-locally}, \cite{defoha-separ}, \cite{defoha-polyhedral},
and the final solution  in \cite{HaTa}.

We refer to the monographs \cite{dgz} and \cite{HJ book}
for a more complete discussion and references, too numerous to
be included in our note.

\medskip{}

The  main result of the present note delves deeper into the fine
behaviour of $C^k$-smooth approximations of norms in the separable setting.
It is in some sense analogous to the condition (ii) in Theorem VIII.3.2 in \cite{dgz},
which claims that in the Banach space with $C^k$-smooth partitions of unity,
the $C^k$-smooth approximations to continuous functions exist
with a prescribed precision around each point. 
Our result solves Problem 170 (stated somewhat imprecisely) in \cite{gmz}.
We also hope that the result may be of some use in the context of metric fixed 
point theory, where several notions are present of properties which asymptotically
improve with growing codimension. For example, let us mention the notion of  {\it asymptotically
non-expansive function} or the ones of {\it asymptotically isometric copy} of $\ell_1$ or $c_0$.

Let us now state our main result.

\begin{thm}
\label{thm:Ck norm improving}Let $(X,\|\cdot\|)$ be a separable real Banach space
with a Schauder basis $\left\{ e_{i}\right\} _{i\geq1}$ that
admits a $C^{k}$-smooth renorming. Then for every sequence
$\left\{ \varepsilon_{N}\right\} _{N\geq0}$ of positive numbers,
there is a $C^{k}$-smooth renorming $\left|\left|\left|\cdot\right|\right|\right|$
of $X$ such that for every $N\geq0$
\[
\Bigl|\,\left|\left|\left|x\right|\right|\right|-\left\Vert x\right\Vert \,\Bigr|\leq\varepsilon_{N}\left\Vert x\right\Vert \qquad\text{for }x\in X^{N},
\]
where \textup{$X^{N}:=\overline{\text{span}}\left\{ e_{i}\right\} _{i\geq N+1}$}.
\end{thm}

In other words, we can approximate the original norm with a $C^{k}$-smooth
one in a way that on the "tail vectors" the approximation is improving
as fast as we wish.

The proof of Theorem \ref{thm:Ck norm improving} will be presented
in the next section. The
rough idea is the following.  By the result in \cite{HaTa}, for
every $N$ one can find a $C^{k}$-smooth norm $\left\Vert \cdot\right\Vert _{N}$
such that $\Bigl|\,\left\Vert \cdot\right\Vert _{N}-\left\Vert \cdot\right\Vert \,\Bigr|\leq\varepsilon_{N}\left\Vert \cdot\right\Vert $.
One is tempted to use the standard gluing  together in a $C^{k}$-smooth way
and hope that the resulting norm will be as desired.
Unfortunately, in this way there is no possibility to assure
that on $X^{N}$ only the $\left\Vert \cdot\right\Vert _{n}$ norms
with $n\geq N$ will enter into the gluing procedure. To achieve this feature
 it is necessary that the norms $\left\Vert \cdot\right\Vert _{N}$
be quantitatively different on $X^{N}$ and $X_{N}=\text{span}\left\{ e_{i}\right\} _{i=1}^{N}$.
The first part of the argument, consisting of the geometric Lemma
\ref{lem:main lemma} and some easy deductions, is exactly aimed at
finding new norms which are quantitatively different on tail vectors
and "head vectors''. The second step consists in iterating this
renorming for every $n$ and rescaling the norms. Finally, we suitably
approximate these norms with $C^{k}$-smooth ones and we glue everything
together using the standard technique.

\section{\label{sec:Proof of Th} Proof of the main result}

In this section we shall prove Theorem \ref{thm:Ck norm improving}.

Let $X$ be a separable (real) Banach space with norm $\left\Vert \cdot\right\Vert $
and  a Schauder basis $\left\{ e_{i}\right\} _{i\geq1}$.
We denote by $K:=\text{b.c}.\left\{ e_{i}\right\} _{i\geq1}$ the basis constant of the Schauder basis
(of course $K$ depends on the particular norm we are using). We also
let $P_{k}$ be the usual projection $P_{k}(\sum_{j\geq1}\alpha^{j}e_{j})=\sum_{j=1}^{k}\alpha^{j}e_{j}$
and $P^{k}:=I_{X}-P_{k}$, i.e. $P^{k}(\sum_{j\geq1}\alpha^{j}e_{j})=\sum_{j\geq k+1}\alpha^{j}e_{j}$.
It is clear that  $\left\Vert P_{k}\right\Vert \leq K$ and
$\left\Vert P^{k}\right\Vert \leq K+1$. Finally, we denote $X_{k}:=\text{span}\left\{ e_{i}\right\} _{i=1}^{k}$
and $X^{k}=\overline{\text{span}}\left\{ e_{i}\right\} _{i\geq k+1}$,
i.e. the ranges of the two projections respectively.

We will make extensive use of convex sets: let us recall that a convex
set $C$ in a Banach space $X$ is a \textit{convex body} if it has
nonempty interior. Obviously a symmetric convex body is in particular
a neighborhood of the origin and the unit ball $B_{X}$ of $X$ is
a bounded, symmetric convex body (we shorthand this fact by saying
that it is a BCSB). Any other BCSB $B$ in $X$ induces an equivalent
norm on $X$ via its Minkowski functional
\[
\mu_{B}(x):=\inf\left\{ t>0:x\in tB\right\} .
\]
We will also denote by $\left\Vert \cdot\right\Vert _{B}$ the norm
induced by $B$, i.e. $\left\Vert x\right\Vert _{B}:=\mu_{B}(x)$;
obviously $\left\Vert \cdot\right\Vert _{B_{X}}$ is the original
norm of the space. Moreover we clearly have
\[
B\subseteq C\implies\mu_{B}\geq\mu_{C},
\]
\[
\mu_{\lambda B}=\frac{1}{\lambda}\mu_{B}
\]
 and passing to the induced norms we see that
\[
B\subseteq C\subseteq(1+\delta)B\implies\frac{1}{1+\delta}\left\Vert \cdot\right\Vert _{B}\leq\left\Vert \cdot\right\Vert _{C}\leq\left\Vert \cdot\right\Vert _{B}.
\]

We now start with the first part of the argument.

\begin{lem}
\label{lem:main lemma}Let $(X,\left\Vert \cdot\right\Vert )$ be
a Banach space with
a Schauder basis $\left\{ e_{i}\right\} _{i\geq1}$ with basis constant
$K$. Denote the unit ball of $X$ by $B$, fix $k\in\mathbb{N}$, two parameters $\lambda>0$ and $0<R<1$,
and consider the sets
\[
D:=\left\{ x\in X:\left\Vert P^{k}x\right\Vert \leq R\right\} \cap(1+\lambda)\cdot B,
\]
\textup{
\[
C:=\overline{\text{conv}}\left\{ D,B\right\} .
\]
}Then $C$ is a BCSB and
\[
C\cap X^{k}\subseteq\left(1+\lambda\frac{K}{K+1-R}\right)\cdot B.
\]
\end{lem}

Heuristically, if we modify the unit ball in the direction of $X_{k}$,
this modification results in a perturbation of the ball also in the
remaining directions, but this modification is significantly smaller.

\begin{proof}
The fact that $C$ is a BCSB is obvious. Let $x\in C\cap X^{k}$ and
notice that $0\in\text{Int}C$ (as $B\subseteq C$); by the cone argument
we deduce that $tx\in\text{Int}C$ for $t\in[0,1)$. Moreover $\text{conv}\left\{ D,B\right\} $
has non-empty interior, so it is easily seen that its interior equals
the interior of its closure, hence $tx\in\text{Int}C=\text{Int}\left(\text{conv}\left\{ D,B\right\} \right)\subseteq\text{conv}\left\{ D,B\right\} $.
If we can show that $tx\in\left(1+\lambda\frac{K}{K+1-R}\right)\cdot B$
we then let $t\rightarrow1$ and conclude the proof. In other words
we can assume without loss of generality that $x\in X^{k}\cap\text{conv}\left\{ D,B\right\} .$
Hence we can write $x=ty+(1-t)z$ with $t\in[0,1]$, $y\in D$ and
$z\in B$, in particular $\left\Vert P^{k}y\right\Vert \leq R$ and
$\left\Vert z\right\Vert \leq1$. Moreover $x\in X^{k}$ implies
\[
\left\Vert x\right\Vert =\left\Vert P^{k}x\right\Vert \leq t\left\Vert P^{k}y\right\Vert +(1-t)\left\Vert P^{k}z\right\Vert \leq tR+(1-t)(K+1);
\]
if $\left\Vert x\right\Vert \leq1$ the conclusion of the lemma is
clearly true, so we can assume $\left\Vert x\right\Vert \geq1$. Thus
we have $1\leq K+1-t(K+1-R)$ and this yields $t\leq\frac{K}{K+1-R}$.

Next, we move the points $y,z$ slightly, in such a way that $x$ is still a
convex combination of them: fix two parameters $\tau,\eta>0$ to be
chosen later and consider $u:=(1-\tau)y$ and $v:=(1+\eta)z$. Obviously
$x=\frac{t}{1-\tau}u+\frac{1-t}{1+\eta}v$ and we require this to
be a convex combination:
\[
1=\frac{t}{1-\tau}+\frac{1-t}{1+\eta}\qquad\implies\qquad\tau=\frac{(1-t)\eta}{t+\eta}\leq1
\]
(of course this choice implies $1-\tau\geq0$). Since $y\in D$, we
have $\left\Vert v\right\Vert \leq1+\eta$ and $\left\Vert u\right\Vert \leq(1-\tau)\left\Vert y\right\Vert \leq(1-\tau)(1+\lambda)$;
we want these norms to be both small, so we require (here we use the
previous choice of $\tau$)
\[
1+\eta=(1-\tau)(1+\lambda)\qquad\implies\qquad\eta=\lambda t.
\]
With this choice of $\tau$ and $\eta$ we have $\left\Vert u\right\Vert, \left\Vert v\right\Vert \leq1+\eta=1+\lambda t\leq1+\lambda\cdot\frac{K}{K+1-R}$;
by convexity the same holds true for $x$ and the proof is complete.
\end{proof}

We now modify again the obtained BCSB in such a way that on $X^{k}$
the body is an exact multiple of the original ball; this modification
does not destroy the properties achieved before. It will be useful
to denote by $S:=\left\{ x\in X:\left\Vert P^{k}x\right\Vert \leq R\right\} $;
with this notation we have $D:=S\cap(1+\lambda)\cdot B$.

\begin{cor}
In the above setting, let $\gamma:=\frac{K}{K+1-R}$ and \textup{
\[
\tilde{B}:=\overline{\text{conv}}\left\{ C,X^{k}\cap(1+\lambda\gamma)\cdot B\right\} .
\]
}Then $\tilde{B}$ is a BCSB and
\[
B\subseteq\tilde{B}\subseteq(1+\lambda)\cdot B,
\]
\[
S\cap\tilde{B}=S\cap(1+\lambda)\cdot B,
\]
\[
X^{k}\cap\tilde{B}=X^{k}\cap(1+\lambda\gamma)\cdot B.
\]
\end{cor}

\begin{proof}
It is obvious that $\tilde{B}$ is a BCSB. Of course $B\subseteq C$,
so $B\subseteq\tilde{B}$ too; also $D\subseteq(1+\lambda)\cdot B$
implies $C\subseteq(1+\lambda)\cdot B$. Since $\gamma\leq1$ we deduce
that $\tilde{B}\subseteq(1+\lambda)\cdot B$.

The $\subseteq$ in the second assertion follows from what we have
just proved; for the converse inclusion, just observe that $S\cap(1+\lambda)\cdot B=D\subseteq\tilde{B}$.

For the last equality, obviously $X^{k}\cap(1+\lambda\gamma)\cdot B\subseteq\tilde{B}$,
so the $\supseteq$ inclusion follows. For the converse inclusion,
let $p\in X^{k}\cap\tilde{B}$; exactly the same argument as in the
first part of the previous proof (with $C$ replaced by $\tilde{B}$)
shows that we can assume $p\in\text{conv}\left\{ C,X^{k}\cap(1+\lambda\gamma)\cdot B\right\} \cap X^{k}$.
So we can write $p=ty+(1-t)z$ with $y\in C$ and $z\in X^{k}\cap(1+\lambda\gamma)\cdot B$.
If $t=0$, $p=z\in X^{k}\cap(1+\lambda\gamma)\cdot B$ and we are
done. On the other hand if $t>0$, from $p\in X^{k}$ we deduce that
$y\in X^{k}$ too; hence in fact $y\in C\cap X^{k}\subseteq(1+\lambda\gamma)\cdot B$,
by the previous lemma. By convexity $p\in(1+\lambda\gamma)\cdot B$
and the proof is complete.
\end{proof}

The next proposition is essentially a restatement of the above corollary
in terms of norms rather than convex bodies; we write it explicitly
since in what follows we will use the approach using norms. The
general setting is the one above: we have a separable Banach space
$X$ with a Schauder basis $\left\{ e_{i}\right\} _{i\geq1}$ and
we denote by $X_{k}:=\text{span}\left\{ e_{i}\right\} _{i=1}^{k}$
and $X^{k}=\overline{\text{span}}\left\{ e_{i}\right\} _{i\geq k+1}$.

\begin{prop}
Let $B$ be a BCSB in $X$ and let $\left\Vert \cdot\right\Vert _{B}$
be the induced norm; also let $K$ be the basis constant of $\left\{ e_{i}\right\} _{i\geq1}$
relative to $\left\Vert \cdot\right\Vert _{B}$. Fix $k\in\mathbb{N}$
and two parameters $\lambda>0$ and $0<R<1$. Then there is a BCSB
$\tilde{B}$ in $X$ such that the induced norm $\left\Vert \cdot\right\Vert _{\tilde{B}}$
satisfies the following properties:
\begin{description}
\item [{\textmd{(a)}}]
\[
\left\Vert \cdot\right\Vert _{\tilde{B}}\leq\left\Vert \cdot\right\Vert _{B}\leq(1+\lambda)\left\Vert \cdot\right\Vert _{\tilde{B}},
\]

\item [{\textmd{(b)}}]
\[
\left\Vert \cdot\right\Vert _{B}=(1+\lambda\gamma)\left\Vert \cdot\right\Vert _{\tilde{B}}\qquad\text{on }X^{k},
\]

\item [{\textmd{(c)}}]
\[
\left\Vert x\right\Vert _{B}=(1+\lambda)\left\Vert x\right\Vert _{\tilde{B}}\qquad\text{whenever }\left\Vert P^{k}x\right\Vert \leq\frac{R}{1+\lambda}\left\Vert x\right\Vert ,
\]
\end{description}
where $\gamma:=\frac{K}{K+1-R}$.
\end{prop}

\begin{proof}
We let $\tilde{B}$ be the convex body defined in the corollary. Then
(a) follows immediately from the corollary and (b) is immediate too:
for $x\in X^{k}$
\[
\left\Vert x\right\Vert _{\tilde{B}}=\inf\left\{ t>0:x\in t\cdot\tilde{B}\right\} =\inf\left\{ t>0:x\in t\cdot\left(\tilde{B}\cap X^{k}\right)\right\}=
\]
\[
\inf\left\{ t>0:x\in t\cdot\left(X^{k}\cap\left(1+\lambda\gamma\right)\cdot B\right)\right\} =\inf\left\{ t>0:x\in t\left(1+\lambda\gamma\right)\cdot B\right\}
\]
\[
=\frac{1}{1+\lambda\gamma}\inf\left\{ t>0:x\in t\cdot B\right\} =\frac{1}{1+\lambda\gamma}\left\Vert x\right\Vert _{B}.
\]

The last equality is not completely trivial since $S$ is not a cone,
so we first modify it and we define
\[
S_{1}:=\left\{ x\in X:\left\Vert P^{k}x\right\Vert \leq\frac{R}{1+\lambda}\left\Vert x\right\Vert \right\} .
\]
We observe that replacing $S$ with $S_{1}$ does not modify the construction:
if we set $D_{1}:=S_{1}\cap(1+\lambda)\cdot B$, then we have $C_{1}:=\overline{\text{conv}}\left\{ D_{1},B\right\} =C$.
In fact $S_{1}\cap(1+\lambda)\cdot B\subseteq S$ implies $C_{1}\subseteq C$
and the converse inclusion follows from $D\subseteq\text{conv}\left\{ D_{1},B\right\} $.
In order to prove this, fix $x\in D$; then $\left\Vert P^{k}x\right\Vert \leq R<1$
and in particular $P^{k}x\in B$. Now set $x_{t}:=P^{k}x+t(x-P^{k}x)$
and choose $t\geq1$ such that $\left\Vert x_{t}\right\Vert =1+\lambda$;
with this choice of $t$ we get $\left\Vert P^{k}x_{t}\right\Vert =\left\Vert P^{k}x\right\Vert \leq R=\frac{R}{1+\lambda}\left\Vert x_{t}\right\Vert $,
so $x_{t}\in D_{1}$. Since $x$ is a convex combination of $x_{t}$
and $P^{k}x$ we deduce $D\subseteq\text{conv}\left\{ D_{1},B\right\} $.

Next, we claim that
\[
S_{1}\cap\tilde{B}=S_{1}\cap(1+\lambda)\cdot B.
\]
In fact $\supseteq$ follows from the analogous relation with $S$,
proved in the corollary, and $S_{1}\cap(1+\lambda)\cdot B\subseteq S$.
The converse inclusion follows from the usual $\tilde{B}\subseteq(1+\lambda)\cdot B$.

Finally we prove (c): pick $x\in S_{1}$ and notice that
\[
\left\{ t>0:x\in t\tilde{B}\right\} =\left\{ t>0:x\in t\tilde{B}\cap S_{1}\right\} =\left\{ t>0:x\in t\left(\tilde{B}\cap S_{1}\right)\right\}
\]
\[
=\left\{ t>0:x\in t\left(S_{1}\cap\left(1+\lambda\right)\cdot B\right)\right\} =\left\{ t>0:x\in t\left(1+\lambda\right)\cdot B\right\} ;
\]
hence
\[
\inf\left\{ t>0:x\in t\tilde{B}\right\} =\frac{1}{1+\lambda}\inf\left\{ t>0:x\in t\cdot B\right\} ,
\]
 which is exactly (c).
\end{proof}

We now iterate the above renorming procedure: we start with the Banach
space $X$ with unit ball $B$ and corresponding norm $\left\Vert \cdot\right\Vert :=\left\Vert \cdot\right\Vert _{B}$
and we apply the proposition with $k=1$, a certain $\lambda_{1}>0$
and $R=1/2$. We let $B_{1}:=\widetilde{B}$ be the obtained body
and $\left\Vert \cdot\right\Vert _{1}:=\left\Vert \cdot\right\Vert _{B_{1}}$
be the corresponding norm. Then we have

\[
\left\Vert \cdot\right\Vert _{1}\leq\left\Vert \cdot\right\Vert \leq(1+\lambda_{1})\left\Vert \cdot\right\Vert _{1},
\]

\[
\left\Vert \cdot\right\Vert =(1+\lambda_{1}\gamma_{1})\left\Vert \cdot\right\Vert _{1}\qquad\text{on }X^{1},
\]

\[
\left\Vert x\right\Vert =(1+\lambda_{1})\left\Vert x\right\Vert _{1}\qquad\text{whenever }\left\Vert P^{1}x\right\Vert \leq\frac{1/2}{1+\lambda_{1}}\left\Vert x\right\Vert ,
\]

where $\gamma_{1}:=\frac{K}{K+1/2}$.

We proceed inductively in the obvious way: fix a sequence $\left\{ \lambda_{n}\right\} _{n\geq1}\subseteq(0,\infty)$
such that $\prod_{i=1}^{\infty}(1+\lambda_{i})<\infty$ and, in order
to have a more concise notation, denote by $\left\Vert \cdot\right\Vert _{0}:=\left\Vert \cdot\right\Vert $
the original norm of $X$ and by $K_{0}:=K$. Apply inductively the
previous proposition: at the step $n$ we use the proposition with
$\lambda=\lambda_{n}$, $R=1/2$, $k=n$ and $B=B_{n-1}$ and we set
$B_{n}:=\widetilde{B_{n-1}}$ and $\left\Vert \cdot\right\Vert _{n}:=\left\Vert \cdot\right\Vert _{B_{n}}$.
This gives a sequence of norms $\left\{ \left\Vert \cdot\right\Vert _{n}\right\} _{n\geq0}$
on $X$ such that for every $n\geq1$ we have:

\begin{equation}
\left\Vert \cdot\right\Vert _{n}\leq\left\Vert \cdot\right\Vert _{n-1}\leq(1+\lambda_{n})\left\Vert \cdot\right\Vert _{n},\label{eq: equiv norm}
\end{equation}

\begin{equation}
\left\Vert \cdot\right\Vert _{n-1}=(1+\lambda_{n}\gamma_{n})\left\Vert \cdot\right\Vert _{n}\qquad\text{on }X^{n},\label{eq: rel on tails}
\end{equation}

\begin{equation}
\left\Vert x\right\Vert _{n-1}=(1+\lambda_{n})\left\Vert x\right\Vert _{n}\qquad\text{whenever }\left\Vert P^{n}x\right\Vert _{n-1}\leq\frac{1/2}{1+\lambda_{n}}\left\Vert x\right\Vert _{n-1},\label{eq: rel with projection}
\end{equation}

where $K_{n}$ denotes the basis constant of $\left\{ e_{i}\right\} _{i\geq1}$
relative to $\left\Vert \cdot\right\Vert _{n}$ and $\gamma_{n}:=\frac{K_{n-1}}{K_{n-1}+1/2}\in(0,1)$.

\begin{rem}
The condition $\left\Vert P^{n}x\right\Vert _{n-1}\leq\frac{1/2}{1+\lambda_{n}}\left\Vert x\right\Vert _{n-1}$
appearing in (\ref{eq: rel with projection}) is somewhat unpleasing
since the involved norms change with $n$; we thus replace it with
the following more uniform, but weaker, condition.

\begin{equation}
\left\Vert x\right\Vert _{n-1}=(1+\lambda_{n})\left\Vert x\right\Vert _{n}\qquad\text{whenever }\left\Vert P^{n}x\right\Vert _{0}\leq\frac{1}{2}\prod_{i=1}^{\infty}(1+\lambda_{i})^{-1}\cdot\left\Vert x\right\Vert _{0}.\label{eq: uniform rel with projection}
\end{equation}

The validity of (\ref{eq: uniform rel with projection}) is immediately
deduced from the validity of (\ref{eq: equiv norm}) and (\ref{eq: rel with projection}):
in fact if $x$ satisfies $\left\Vert P^{n}x\right\Vert _{0}\leq\frac{1}{2}\prod_{i=1}^{\infty}(1+\lambda_{i})^{-1}\cdot\left\Vert x\right\Vert _{0}$,
then by (\ref{eq: equiv norm})
\[
\left\Vert P^{n}x\right\Vert _{n-1}\leq\left\Vert P^{n}x\right\Vert _{0}\leq\frac{1}{2}\prod_{i=1}^{\infty}(1+\lambda_{i})^{-1}\cdot\left\Vert x\right\Vert _{0}\leq\frac{1}{2}\prod_{i=1}^{\infty}(1+\lambda_{i})^{-1}\cdot\prod_{i=1}^{n-1}(1+\lambda_{i})\cdot\left\Vert x\right\Vert _{n-1}
\]
\[
=\frac{1/2}{1+\lambda_{n}}\prod_{i=n+1}^{\infty}(1+\lambda_{i})^{-1}\cdot\left\Vert x\right\Vert _{n-1}\leq\frac{1/2}{1+\lambda_{n}}\left\Vert x\right\Vert _{n-1};
\]
hence (\ref{eq: rel with projection}) implies that $\left\Vert x\right\Vert _{n-1}=(1+\lambda_{n})\left\Vert x\right\Vert _{n}$.
\end{rem}

In order to motivate the next step, let us notice that for a fixed
$x\in X$ the sequence $\left\{ \left\Vert x\right\Vert _{n}\right\} _{n\geq0}$
has the same qualitative behavior since it is a decreasing sequence;
on the other hand the quantitative rate of decrease changes with $n$.
In fact it is clear that for a fixed $x\in X$, the condition $\left\Vert P^{n}x\right\Vert _{0}\leq\frac{1}{2}\prod_{i=1}^{\infty}(1+\lambda_{i})^{-1}\cdot\left\Vert x\right\Vert _{0}$
is eventually satisfied, so the sequence $\left\{ \left\Vert x\right\Vert _{n}\right\} _{n\geq0}$
eventually decreases with rate $(1+\lambda_{n})^{-1}$. On the other
hand if $x\in X^{N}$, then for the terms $n=1,\dots,N$ the rate
of decrease is $(1+\lambda_{n}\gamma_{n})^{-1}$. 
This makes it possible to rescale the norms $\left\{ \left\Vert \cdot\right\Vert _{n}\right\} _{n\geq0}$,
obtaining norms $\left\{ \left|\left|\left|\cdot\right|\right|\right|_{n} \right\} _{n\geq0}$,
in a way to have a qualitatively different behavior, increasing for $n=1,\dots,N$ and eventually decreasing.
This property is crucial since it allows us to assure that, for $x\in X^N$, the norms $\left|\left|\left|x\right|\right|\right|_{n}$
for $n=0,\dots ,N-1$ are quantitatively smaller than $\left|\left|\left|x\right|\right|\right|_{N}$ and thus 
do not enter in the gluing procedure. As we have hinted at at the end of the previous section and as it will be apparent
in the proof of Lemma \ref{lem:||| is equiv norm}, this is exactly what we need in order the approximation
on $X^N$ to improve with $N$.

\begin{defn}
Let
\[
C:=\prod_{i=1}^{\infty}\frac{1+\lambda_{i}\gamma_{i}}{1+\lambda_{i}\frac{1+\gamma_{i}}{2}},
\]
\[
\left|\left|\left|\cdot\right|\right|\right|_{n}:=C\cdot\prod_{i=1}^{n}\left(1+\lambda_{i}\frac{1+\gamma_{i}}{2}\right)\cdot\left\Vert \cdot\right\Vert _{n}.
\]
For later convenience, let us also set
\[
\left|\left|\left|\cdot\right|\right|\right|_{\infty}=\sup_{n\geq0}\left|\left|\left|\cdot\right|\right|\right|_{n}.
\]
\end{defn}

The qualitative behavior of $\left\{ \left|\left|\left|x\right|\right|\right|_{n}\right\} _{n\geq0}$
is expressed in the following obvious, though crucial, properties
of the norms $\left|\left|\left|\cdot\right|\right|\right|_{n}$.
In particular, (a) will be used to show that the gluing together locally takes into account
only finitely many terms; this will allow us to preserve the smoothness in Lemma \ref{lem:smooth norm}.
(b) expresses the fact that on $X^N$ the norms $\left\{ \left|\left|\left|\cdot\right|\right|\right|_{n} \right\} _{n=0}^{N-1}$
are smaller than  $\left|\left|\left|\cdot\right|\right|\right|_{N}$ and will be used in Lemma \ref{lem:||| is equiv norm}
to obtain the improvement of the approximation.

\begin{fact}
\label{Fact propr of |||.|||} \textup{(a)} For every $x\in X$ there
is $n_{0}\in\mathbb{N}$ such that for every $n\geq n_{0}$
\[
\left|\left|\left|x\right|\right|\right|_{n}=\frac{1+\lambda_{n}\frac{1+\gamma_{n}}{2}}{1+\lambda_{n}}\left|\left|\left|x\right|\right|\right|_{n-1}.
\]
In particular, it suffices to take any $n_{0}$ such that $\left\Vert P^{n}x\right\Vert _{0}\leq\frac{1}{2}\prod_{i=1}^{\infty}(1+\lambda_{i})^{-1}\cdot\left\Vert x\right\Vert _{0}$
for every $n\geq n_{0}$.

\textup{(b)} If $x\in X^{N}$, then for $n=1,\dots,N$ we have
\[
\left|\left|\left|x\right|\right|\right|_{n}=\frac{1+\lambda_{n}\frac{1+\gamma_{n}}{2}}{1+\lambda_{n}\gamma_{n}}\left|\left|\left|x\right|\right|\right|_{n-1}.
\]
\end{fact}

\begin{proof}
(a) Since $P^{n}x\rightarrow0$ as $n\rightarrow\infty$, condition
(\ref{eq: uniform rel with projection}) implies that there is $n_{0}$
such that for every $n\geq n_{0}$ we have $\left\Vert x\right\Vert _{n}=(1+\lambda_{n})^{-1}\left\Vert x\right\Vert _{n-1}$.
Then it suffices to translate this to the $\left|\left|\left|\cdot\right|\right|\right|_{n}$
norms:
\[
\left|\left|\left|x\right|\right|\right|_{n}=\left(1+\lambda_{n}\frac{1+\gamma_{n}}{2}\right)\cdot C\cdot\prod_{i=1}^{n-1}\left(1+\lambda_{i}\frac{1+\gamma_{i}}{2}\right)\cdot\left\Vert x\right\Vert _{n}=
\]
\[
\frac{1+\lambda_{n}\frac{1+\gamma_{n}}{2}}{1+\lambda_{n}}\cdot C\cdot\prod_{i=1}^{n-1}\left(1+\lambda_{i}\frac{1+\gamma_{i}}{2}\right)\cdot\left\Vert x\right\Vert _{n-1}=\frac{1+\lambda_{n}\frac{1+\gamma_{n}}{2}}{1+\lambda_{n}}\left|\left|\left|x\right|\right|\right|_{n-1}.
\]
(b) If $x\in X^{N}$ and $n=1,\dots,N$, then $x\in X^{n}$ too;
thus by (\ref{eq: rel on tails}) we have $\left\Vert x\right\Vert _{n}=(1+\lambda_{n}\gamma_{n})^{-1}\left\Vert x\right\Vert _{n-1}$.
Now exactly the same calculation as in the other case gives the result.
\end{proof}

We can now conclude the renorming procedure: first we smoothen up
the norms $\left|\left|\left|\cdot\right|\right|\right|_{n}$ and
then we glue together all these smooth norms. Fix a decreasing sequence
$\delta_{n}\searrow0$ such that for every $n\geq0$
\[
(\dagger)\qquad(1+\delta_{n})\frac{1+\lambda_{n+1}\gamma_{n+1}}{1+\lambda_{n+1}\frac{1+\gamma_{n+1}}{2}}\leq1-\delta_{n}
\]
 (of course this is possible since $\gamma_{n+1}<1$). Then we apply
the main result in \cite{HaTa} (Theorem 2.10 in their paper) to
find $C^{k}$-smooth norms $\left\{ \left|\left|\left|\cdot\right|\right|\right|_{(s),n}\right\} _{n\geq0}$
such that for every $n$
\[
\left|\left|\left|\cdot\right|\right|\right|_{n}\leq\left|\left|\left|\cdot\right|\right|\right|_{(s),n}\leq(1+\delta_{n})\left|\left|\left|\cdot\right|\right|\right|_{n}.
\]

Next, let $\varphi_{n}:[0,\infty)\rightarrow[0,\infty)$ be $C^{\infty}$-smooth,
convex and such that $\varphi_{n}\equiv0$ on $[0,1-\delta_{n}]$
and $\varphi_{n}(1)=1$; note that of course the $\varphi_{n}$'s
are strictly monotonically increasing on $[1-\delta_{n},\infty)$. Finally
define $\Phi:X\rightarrow[0,\infty]$ by
\[
\Phi(x):=\sum_{n\geq0}\varphi_{n}\left(\left|\left|\left|x\right|\right|\right|_{(s),n}\right)
\]
and let $\left|\left|\left|\cdot\right|\right|\right|$ be the Minkowski
functional of the set $\left\{ \Phi\leq1\right\} $.

\medskip{}

The fact that $\left|\left|\left|\cdot\right|\right|\right|$ is the
desired norm is now an obvious consequence of the next two lemmas.
In the first one we show that $\left|\left|\left|\cdot\right|\right|\right|$
is indeed a norm and that the approximation on $X^{N}$ improves with $N$.

\begin{lem}
\label{lem:||| is equiv norm}$\left|\left|\left|\cdot\right|\right|\right|$
is a norm, equivalent to the original norm $\left\Vert \cdot\right\Vert $
of $X$. 

Moreover for every $N\geq0$ we have
\[
\prod_{i=N+1}^{\infty}\left(1+\lambda_{i}\right)^{-1}\cdot\left\Vert \cdot\right\Vert 
\leq\left|\left|\left|\cdot\right|\right|\right|\leq
\frac{1+\delta_{N}}{1-\delta_{N}}\cdot\prod_{i=N+1}^{\infty}\left(1+\lambda_{i}\right)\cdot\left\Vert \cdot\right\Vert
\qquad\text{on }X^{N}.
\]
\end{lem}

\begin{proof}
We start by observing that for every $N\geq0$
\[
\left\{ x\in X^{N}:\left|\left|\left|x\right|\right|\right|_{\infty}\leq\frac{1-\delta_{N}}{1+\delta_{N}}\right\}
\subseteq\left\{ x\in X^{N}:\Phi(x)\leq1\right\}\subseteq
\left\{ x\in X^{N}:\left|\left|\left|x\right|\right|\right|_{\infty}\leq1\right\}.
\]
In fact, pick $x\in X^{N}$ such that $\Phi(x)\leq1$, so in particular $\varphi_{n}\left(\left|\left|\left|x\right|\right|\right|_{(s),n}\right)\leq1$
for every $n$. The inequality $\left|\left|\left|\cdot\right|\right|\right|_{n}\leq\left|\left|\left|\cdot\right|\right|\right|_{(s),n}$
and the properties of $\varphi_{n}$ then imply $\left|\left|\left|x\right|\right|\right|_{n}\leq1$
for every $n$. This proves the right inclusion. For the first inclusion, we actually
show that if $x\in X^N$ satisfies $\left|\left|\left|x\right|\right|\right|_{\infty}
\leq\frac{1-\delta_{N}}{1+\delta_{N}}$, then $\Phi(x)=0$.
To see this, fix any $n\geq N$; since the function $t\mapsto\frac{1-t}{1+t}$ is decreasing on $[0,1]$
and the sequence $\delta_{n}$ is decreasing too, we deduce
\[
\left|\left|\left|x\right|\right|\right|_{n}\leq\left|\left|\left|x\right|\right|\right|_{\infty}
\leq\frac{1-\delta_{N}}{1+\delta_{N}}\leq\frac{1-\delta_{n}}{1+\delta_{n}}.
\]
Hence $\left|\left|\left|x\right|\right|\right|_{(s),n}\leq1-\delta_{n}$
and $\varphi_{n}\left(\left|\left|\left|x\right|\right|\right|_{(s),n}\right)=0$
for every $n\geq N$. For the remaining values $n=0,\dots,N-1$
we use (b) in Fact \ref{Fact propr of |||.|||} and condition $(\dagger)$:
\[
\left|\left|\left|x\right|\right|\right|_{(s),n}\leq(1+\delta_{n})\left|\left|\left|x\right|\right|\right|_{n}=
(1+\delta_{n})\frac{1+\lambda_{n+1}\gamma_{n+1}}{1+\lambda_{n+1}\frac{1+\gamma_{n+1}}{2}}
\cdot\left|\left|\left|x\right|\right|\right|_{n+1}
\]
\[
\leq(1-\delta_{n})\left|\left|\left|x\right|\right|\right|_{n+1}\leq1-\delta_{n};
\]
hence $\varphi_{n}\left(\left|\left|\left|x\right|\right|\right|_{(s),n}\right)=0$
for $n=0,\dots,N-1$ too. This implies $\Phi(x)=0$ and proves the first
inclusion.

Taking in particular $N=0$, we see that $\left\{ \Phi\leq1\right\}$ is a bounded neighborhood
of the origin in $(X,\left|\left|\left|\cdot\right|\right|\right|_{\infty})$. Since it is clearly convex and symmetric,
we deduce that $\left\{ \Phi\leq1\right\} $ is a BCSB relative to $\left|\left|\left|\cdot\right|\right|\right|_{\infty}$.
Hence $\left|\left|\left|\cdot\right|\right|\right|$ is a norm on $X$, equivalent to $\left|\left|\left|\cdot\right|\right|\right|_{\infty}$.
The fact that $\left|\left|\left|\cdot\right|\right|\right|$ is equivalent to the original norm $\left\Vert \cdot\right\Vert$
follows immediately from the case $N=0$ in the second assertion, which we now prove.

\medskip{}

Fix $N\geq 0$; in order to estimate the distortion between $\left|\left|\left|\cdot\right|\right|\right|$
and $\left\Vert \cdot\right\Vert $ on $X^N$, we show that, on $X^N$, $\left|\left|\left|\cdot\right|\right|\right|$
is close to $\left|\left|\left|\cdot\right|\right|\right|_{\infty}$,
that $\left|\left|\left|\cdot\right|\right|\right|_{\infty}$ is close
to $\left|\left|\left|\cdot\right|\right|\right|_{N}$ and finally
that $\left|\left|\left|\cdot\right|\right|\right|_{N}$ is close
to $\left\Vert \cdot\right\Vert $.

First, passing to the associated Minkowski functionals, the above inclusions yield
\[
(*)\qquad\left|\left|\left|\cdot\right|\right|\right|_{\infty}\leq
\left|\left|\left|\cdot\right|\right|\right|\leq
\frac{1+\delta_{N}}{1-\delta_{N}}\left|\left|\left|\cdot\right|\right|\right|_{\infty}
\qquad\text{on }X^{N}.
\]

Next, we compare $\left|\left|\left|\cdot\right|\right|\right|_{\infty}$
with $\left|\left|\left|\cdot\right|\right|\right|_{N}$. Of course
$\left|\left|\left|\cdot\right|\right|\right|_{N}\leq\left|\left|\left|\cdot\right|\right|\right|_{\infty}$
and by property (b) in Fact \ref{Fact propr of |||.|||} already
used above we also have $\left|\left|\left|\cdot\right|\right|\right|_{n}\leq\left|\left|\left|\cdot\right|\right|\right|_{N}$
for $n\leq N$. We thus fix $n>N$ and observe
\[
\left|\left|\left|\cdot\right|\right|\right|_{n}:=C\prod_{i=1}^{n}\left(1+\lambda_{i}\frac{1+\gamma_{i}}{2}\right)\cdot\left\Vert \cdot\right\Vert _{n}\leq
\]
\[
\prod_{i=N+1}^{n}\left(1+\lambda_{i}\frac{1+\gamma_{i}}{2}\right)\cdot C\cdot\prod_{i=1}^{N}\left(1+\lambda_{i}\frac{1+\gamma_{i}}{2}\right)\cdot\left\Vert \cdot\right\Vert _{N}=
\]
\[
\prod_{i=N+1}^{n}\left(1+\lambda_{i}\frac{1+\gamma_{i}}{2}\right)\cdot\left|\left|\left|\cdot\right|\right|\right|_{N}\leq\prod_{i=N+1}^{\infty}\left(1+\lambda_{i}\right)\cdot\left|\left|\left|\cdot\right|\right|\right|_{N}.
\]
This yields
\[
(*)\qquad\left|\left|\left|\cdot\right|\right|\right|_{N}\leq\left|\left|\left|\cdot\right|\right|\right|_{\infty}\leq\prod_{i=N+1}^{\infty}\left(1+\lambda_{i}\right)\cdot\left|\left|\left|\cdot\right|\right|\right|_{N}\qquad\text{on }X^{N}.
\]

Finally, we compare $\left|\left|\left|\cdot\right|\right|\right|_{N}$
with $\left\Vert \cdot\right\Vert _{0}$. The subspaces $X^{N}$ are
decreasing, so (\ref{eq: rel on tails}) implies $\left\Vert \cdot\right\Vert =\prod_{i=1}^{N}\left(1+\lambda_{i}\gamma_{i}\right)\cdot\left\Vert \cdot\right\Vert _{N}$
on $X^{N}$; hence
\[
\left\Vert \cdot\right\Vert =\prod_{i=1}^{N}\left(1+\lambda_{i}\gamma_{i}\right)\cdot\prod_{i=1}^{\infty}\frac{1+\lambda_{i}\frac{1+\gamma_{i}}{2}}{1+\lambda_{i}\gamma_{i}}\cdot\prod_{i=1}^{N}\left(1+\lambda_{i}\frac{1+\gamma_{i}}{2}\right)^{-1}\cdot\left|\left|\left|\cdot\right|\right|\right|_{N}
\]
\[
=\prod_{i=N+1}^{\infty}\frac{1+\lambda_{i}\frac{1+\gamma_{i}}{2}}{1+\lambda_{i}\gamma_{i}}\cdot\left|\left|\left|\cdot\right|\right|\right|_{N}.
\]
This implies in particular
\[
(*)\qquad\left|\left|\left|\cdot\right|\right|\right|_{N}\leq\left\Vert \cdot\right\Vert \leq\prod_{i=N+1}^{\infty}\left(1+\lambda_{i}\right)\cdot\left|\left|\left|\cdot\right|\right|\right|_{N}\qquad\text{on }X^{N};
\]
combining the $(*)$ inequalities concludes the proof of the lemma.
\end{proof}

\begin{rem}
The estimate of the distortion in the particular case $N=0$ is in fact shorter than the general case given above.
In fact, property (\ref{eq: equiv norm}) obviously implies $\left\Vert \cdot\right\Vert _{n}\leq\left\Vert \cdot\right\Vert \leq\prod_{i=1}^{n}(1+\lambda_{i})\cdot\left\Vert \cdot\right\Vert _{n}$.
It easily follows that for every $n$
\[
\prod_{i=1}^{\infty}\left(1+\lambda_{i}\right)^{-1}\cdot\left\Vert \cdot\right\Vert
\leq \left|\left|\left|\cdot\right|\right|\right|_{n}\leq
\prod_{i=1}^{\infty}\left(1+\lambda_{i}\right)\cdot\left\Vert \cdot\right\Vert;
\]
it is then sufficient to combine this with the first of the $(*)$ inequalities.
\end{rem}

We finally check the regularity of $\left|\left|\left|\cdot\right|\right|\right|$.

\begin{lem}
\label{lem:smooth norm}
The norm $\left|\left|\left|\cdot\right|\right|\right|$ is $C^{k}$-smooth.
\end{lem}

\begin{proof}
We first show that for every $x$ in the set $\left\{ \Phi<2\right\} $
there is a neighborhood $\mathcal{U}$ of $x$ (in $X$) where the
function $\Phi$ is expressed by a finite sum. We have already seen
in the proof of Lemma \ref{lem:||| is equiv norm} that $\Phi=0$
in a neighborhood of $0$, so the assertion is true for $x=0$; hence
we can fix $x\neq0$ such that $\Phi(x)<2$. Observe that clearly
the properties of $\varphi_{n}$ imply $\varphi_{n}(1+\delta_{n})\geq2$;
thus $x$ satisfies $\left|\left|\left|x\right|\right|\right|_{n}\leq\left|\left|\left|x\right|\right|\right|_{(s),n}\leq1+\delta_{n}$
for every $n$.

Denote by $c:=\frac{1}{2}\prod_{i=1}^{\infty}(1+\lambda_{i})^{-1}$
and choose $n_{0}$ such that $\left\Vert P^{n}x\right\Vert \leq\frac{c}{2}\cdot\left\Vert x\right\Vert $
for every $n\geq n_{0}$ (this is possible since $P^{n}x\rightarrow0$).
Next, fix $\varepsilon>0$ small so that $\frac{c}{2}+K\varepsilon\leq(1-\varepsilon)c$
and $(1+\varepsilon)(1-\delta_{n_{0}})\leq1$, and let $\mathcal{U}$
be the following neighborhood of $x$:
\[
\mathcal{U}:=\left\{ y\in X:\left\Vert y-x\right\Vert <\varepsilon\left\Vert x\right\Vert \text{ and }\left|\left|\left|y\right|\right|\right|_{n_{0}}<(1+\varepsilon)\left|\left|\left|x\right|\right|\right|_{n_{0}}\right\} .
\]
Clearly for $y\in\mathcal{U}$ we have $\left\Vert x\right\Vert \leq\frac{1}{1-\varepsilon}\left\Vert y\right\Vert $;
thus for $y\in\mathcal{U}$ and $n\geq n_{0}$ we have
\[
\left\Vert P^{n}y\right\Vert \leq\left\Vert P^{n}y-P^{n}x\right\Vert +\left\Vert P^{n}x\right\Vert \leq K\varepsilon\left\Vert x\right\Vert +\frac{c}{2}\cdot\left\Vert x\right\Vert \leq(1-\varepsilon)c\left\Vert x\right\Vert \leq c\left\Vert y\right\Vert .
\]
Hence (a) of Fact \ref{Fact propr of |||.|||} implies that $\left|\left|\left|y\right|\right|\right|_{n}=\frac{1+\lambda_{n}\frac{1+\gamma_{n}}{2}}{1+\lambda_{n}}\left|\left|\left|y\right|\right|\right|_{n-1}$
for every $n\geq n_{0}$ and $y\in\mathcal{U}$ (let us explicitly
stress the crucial fact that $n_{0}$ does not depend on $y\in\mathcal{U}$).

We have $\left|\left|\left|y\right|\right|\right|_{n_{0}}<(1+\varepsilon)\left|\left|\left|x\right|\right|\right|_{n_{0}}\leq(1+\varepsilon)(1+\delta_{n_{0}})$;
using this bound and the previous choices of the parameters (in particular
we use twice $(\dagger)$ and twice the fact that $\delta_{n}$ is
decreasing), for every $n\geq n_{0}+2$ and $y\in\mathcal{U}$ we
estimate
\[
\left|\left|\left|y\right|\right|\right|_{(s),n}\leq(1+\delta_{n})\left|\left|\left|y\right|\right|\right|_{n}=(1+\delta_{n})\prod_{i=n_{0}+1}^{n}\frac{1+\lambda_{i}\frac{1+\gamma_{i}}{2}}{1+\lambda_{i}}\cdot\left|\left|\left|y\right|\right|\right|_{n_{0}}
\]
\[
\leq(1+\delta_{n})\prod_{i=n_{0}+1}^{n}\frac{1+\lambda_{i}\frac{1+\gamma_{i}}{2}}{1+\lambda_{i}}\cdot(1+\varepsilon)(1+\delta_{n_{0}})\overset{(\dagger)}{\leq}(1+\delta_{n})\prod_{i=n_{0}+2}^{n}\frac{1+\lambda_{i}\frac{1+\gamma_{i}}{2}}{1+\lambda_{i}}\cdot(1+\varepsilon)(1-\delta_{n_{0}})
\]
\[
\leq(1+\delta_{n})\prod_{i=n_{0}+2}^{n}\frac{1+\lambda_{i}\frac{1+\gamma_{i}}{2}}{1+\lambda_{i}}\leq(1+\delta_{n-1})\frac{1+\lambda_{n}\frac{1+\gamma_{n}}{2}}{1+\lambda_{n}}\cdot\prod_{i=n_{0}+2}^{n-1}\frac{1+\lambda_{i}\frac{1+\gamma_{i}}{2}}{1+\lambda_{i}}
\]
\[
\leq(1+\delta_{n-1})\frac{1+\lambda_{n}\frac{1+\gamma_{n}}{2}}{1+\lambda_{n}}\overset{(\dagger)}{\leq}1-\delta_{n-1}\leq1-\delta_{n}.
\]
It follows that $\varphi_{n}\left(\left|\left|\left|y\right|\right|\right|_{(s),n}\right)=0$
for $n\geq n_{0}+2$ and $y\in\mathcal{U}$, hence
\[
\Phi=\sum_{n=0}^{n_{0}+2}\varphi_{n}\left(\left|\left|\left|\cdot\right|\right|\right|_{(s),n}\right)\qquad\text{on }\mathcal{U}.
\]
 This obviously implies that $\Phi$ is $C^{k}$-smooth on the set
$\left\{ \Phi<2\right\} $ and in particular $\left\{ \Phi<2\right\} $
is an open set. Concerning the regularity of $\Phi$, we also observe
here that $\Phi$ is lower semi-continuous on $X$ (this follows immediately
from the fact that $\Phi$ is the sum of a series of positive continuous
functions).

\medskip{}

The last step consists in applying the Implicit Function theorem (see e.g. \cite{HJ book}, Theorem 1.87) and
deduce the $C^{k}$-smoothness of $\left|\left|\left|\cdot\right|\right|\right|$
from the one of $\Phi$; this argument is quite well known, but equally
short, so we decided to present it. The set
\[
V:=\left\{ (x,\rho)\in\left(X\backslash\left\{ 0\right\} \right)\times(0,\infty):\rho^{-1}\cdot x\in\left\{ \Phi<2\right\} \right\}
\]
is open in $X\times(0,\infty)$ and the function $\Psi:V\rightarrow\mathbb{R}$
defined by $\Psi(x,\rho):=\Phi(\rho^{-1}\cdot x)$ is $C^{k}$-smooth
on $V$.

We notice that for every $h\in X\backslash\left\{ 0\right\} $ there
is a unique $\rho>0$ such that $(h,\rho)\in V$ and $\Psi(h,\rho)=1$;
moreover, $\rho=\left|\left|\left|h\right|\right|\right|$. In fact
the functions $\varphi_{n}$ are strictly increasing on the set where
they are positive, so $t\mapsto\Phi(th)$ is strictly increasing where
it is positive; hence there is at most one $\rho$ as above. Also,
$\left|\left|\left|h\right|\right|\right|=\inf\left\{ t>0:\Phi(t^{-1}h)\leq1\right\} $,
so for every $\varepsilon>0$ we have $\Phi\left(\frac{1}{\left|\left|\left|h\right|\right|\right|+\varepsilon}h\right)\leq1$;
as $\Phi$ is lower semi-continuous, we deduce $\Phi\left(\left|\left|\left|h\right|\right|\right|^{-1}h\right)\leq1$.
If it were that $\Phi\left(\left|\left|\left|h\right|\right|\right|^{-1}h\right)<1$,
then from the continuity of $\Phi$ on $\left\{ \Phi<2\right\} $
we would deduce $\Phi\left(\frac{1}{\left|\left|\left|h\right|\right|\right|-\varepsilon}h\right)\leq1$
for $\varepsilon>0$ small; however this contradicts $\left|\left|\left|h\right|\right|\right|$
being the infimum. Hence $\Phi\left(\left|\left|\left|h\right|\right|\right|^{-1}h\right)=1$
and in particular the unique $\rho$ as above is $\rho=\left|\left|\left|h\right|\right|\right|$.

In other words, the equation $\Psi=1$ on $V$ globally defines a
unique implicit function on $X\backslash\left\{ 0\right\} $, which
is given by $\rho(h)=\left|\left|\left|h\right|\right|\right|$. Since
\[
D_{2}\Psi(h,\rho)=\frac{-1}{\rho^{2}}\Phi'(\rho^{-1}h)h=\frac{-1}{\rho^{2}}\sum_{n\geq0}\varphi_{n}'\left(\left|\left|\left|\rho^{-1}h\right|\right|\right|_{(s),n}\right)\left|\left|\left|h\right|\right|\right|_{(s),n}
\]
(where $D_{2}\Psi$ denotes the partial derivative of $\Psi$ in its
second variable), we have
\[
D_{2}\Psi(h,\left|\left|\left|h\right|\right|\right|)=\frac{-1}{\left|\left|\left|h\right|\right|\right|^{2}}\sum_{n\geq0}\varphi_{n}'\left(\frac{1}{\left|\left|\left|h\right|\right|\right|}\left|\left|\left|h\right|\right|\right|_{(s),n}\right)\left|\left|\left|h\right|\right|\right|_{(s),n}.
\]
The condition $\Phi\left(\left|\left|\left|h\right|\right|\right|^{-1}h\right)=1$
implies $\varphi_{n}\left(\frac{1}{\left|\left|\left|h\right|\right|\right|}\left|\left|\left|h\right|\right|\right|_{(s),n}\right)>0$
for some $n$, hence $\varphi_{n}'\left(\frac{1}{\left|\left|\left|h\right|\right|\right|}\left|\left|\left|h\right|\right|\right|_{(s),n}\right)>0$
too and $D_{2}\Psi(h,\left|\left|\left|h\right|\right|\right|)\neq0$
on $X\backslash\left\{ 0\right\} $. Thus the Implicit Function theorem
yields that the implicitly defined function shares the same regularity
as $\Psi$, i.e. $\left|\left|\left|\cdot\right|\right|\right|$ is
$C^{k}$-smooth on $X\backslash\left\{ 0\right\} $.
\end{proof}

\begin{proof}
[Proof of Theorem \ref{thm:Ck norm improving}]
Fix a separable Banach space as in the statement and a sequence $\left\{ \varepsilon_{N}\right\} _{N\geq0}$
of positive numbers. We find a sequence $\left\{ \lambda_{i}\right\} _{i\geq1}\subseteq(0,\infty)$
such that
\[
\prod_{i=N+1}^{\infty}(1+\lambda_{i})<1+\varepsilon_{N}
\]
 for every $N\geq0$; next, we find a decreasing sequence $\left\{ \delta_{N}\right\} _{N\geq0}$,
$\delta_{N}\searrow0$, that satisfies $(\dagger)$ and such that
\[
\frac{1+\delta_{N}}{1-\delta_{N}}\cdot\prod_{i=N+1}^{\infty}(1+\lambda_{i})\leq1+\varepsilon_{N}
\]
for every $N\geq0$. We then apply the renorming procedure described
in this section with these parameters $\left\{ \lambda_{i}\right\} _{i\geq1}$
and $\left\{ \delta_{N}\right\} _{N\geq0}$ and we obtain a $C^{k}$-smooth
norm $\left|\left|\left|\cdot\right|\right|\right|$ on $X$ that
satisfies
\[
(1-\varepsilon_{N})\cdot\left\Vert \cdot\right\Vert \leq\prod_{i=N+1}^{\infty}\left(1+\lambda_{i}\right)^{-1}\cdot\left\Vert \cdot\right\Vert \leq\left|\left|\left|\cdot\right|\right|\right|\leq\frac{1+\delta_{N}}{1-\delta_{N}}\cdot\prod_{i=N+1}^{\infty}\left(1+\lambda_{i}\right)\cdot\left\Vert \cdot\right\Vert \leq(1+\varepsilon_{N})\cdot\left\Vert \cdot\right\Vert
\]
 on $X^{N}$ for every $N\geq0$; since these inequalities are obviously
equivalent to
\[
\Bigl|\,\left|\left|\left|x\right|\right|\right|-\left\Vert x\right\Vert \,\Bigr|\leq\varepsilon_{N}\left\Vert x\right\Vert \qquad\text{for }x\in X^{N},
\]
the proof is complete.
\end{proof}

\section{Final remarks}

In this short section we present some improvements
of our main result in the particular case of polyhedral
Banach spaces. Recall that a finite-dimensional Banach space $X$
is said to be \textit{polyhedral}
if its unit ball is a polyhedron, i.e. finite intersection of closed
half-spaces; an infinite-dimensional Banach space $X$ is \textit{polyhedral}
if its finite-dimensional subspaces are polyhedral. It is proved in
\cite{defoha-polyhedral} that if $X$ is a separable polyhedral Banach
space, then every equivalent norm on $X$ can be approximated (uniformly
on bounded sets) by a polyhedral norm (see Theorem 1.1 in \cite{defoha-polyhedral},
where the approximation is stated in terms of closed, convex and bounded bodies).

In analogy with our main result,
it is natural to ask if this result can be improved in the sense that
the approximation can be chosen to be improving on the tail vectors.
It is not difficult to see that if we replace the $C^{k}$-smooth
norms $\left|\left|\left|\cdot\right|\right|\right|_{(s),n}$ with
polyhedral norms $\left|\left|\left|\cdot\right|\right|\right|_{(p),n}$
(thus using Theorem 1.1 in \cite{defoha-polyhedral}) and we replace
the $C^{\infty}$-smooth functions $\varphi_{n}$ with piecewise linear
ones, the resulting norm $\||\cdot\||$ is still polyhedral. We thus have:

\begin{prop}
Let $X$ be a polyhedral Banach space with a Schauder
basis $\left\{ e_{i}\right\} _{i\geq1}$ and let $\left\Vert \cdot\right\Vert $
be any renorming of $X$. Then for every sequence $\left\{ \varepsilon_{N}\right\} _{N\geq0}$
of positive numbers, there is a polyhedral renorming $\left|\left|\left|\cdot\right|\right|\right|$
of $X$ such that for every $N\geq0$ 
\[
\Bigl|\,\left|\left|\left|x\right|\right|\right|-\left\Vert x\right\Vert \,\Bigr|\leq\varepsilon_{N}\left\Vert x\right\Vert \qquad\text{for }x\in X^{N}.
\]
\end{prop}

We say that $\|\cdot\|$ \textit{depends
locally on finitely many coordinates} if for each
$x\in S_{X}$ there exists an open neighbourhood $O$ of $x$, a finite
set $\{x_{1}^{*},\dots,x_{k}^{*}\}\subset X^{*}$ and a function $f:\mathbb{R}^{k}\rightarrow\mathbb{R}$
such that $\|y\|=f(x_{1}^{*}(y),\dots,x_{k}^{*}(y))$ for $y\in O$.
It was also shown in \cite{defoha-polyhedral} that if $X$ is a separable
polyhedral space, then every equivalent norm on $X$ can be approximated
by a $C^{\infty}$-smooth norm that depends locally on finitely many
coordinates. By inspection of our argument it follows that if we use
such approximations in our proof, the resulting $C^{\infty}$-smooth
norm $\||\cdot\||$ will also depend locally on finitely many coordinates.
Explicitly, we obtain:

\begin{prop}
Let $X$ be a polyhedral Banach space with a Schauder
basis $\left\{ e_{i}\right\} _{i\geq1}$ and let $\left\Vert \cdot\right\Vert $
be any renorming of $X$. Then for every sequence $\left\{ \varepsilon_{N}\right\} _{N\geq0}$
of positive numbers, there is a $C^{\infty}$-smooth renorming $\left|\left|\left|\cdot\right|\right|\right|$
of $X$ that locally depends on finitely many coordinates and such
that for every $N\geq0$ 
\[
\Bigl|\,\left|\left|\left|x\right|\right|\right|-\left\Vert x\right\Vert \,\Bigr|\leq\varepsilon_{N}\left\Vert x\right\Vert \qquad\text{for }x\in X^{N}.
\]
\end{prop}

In conclusion of our note, we mention that we do
not know whether our main result can be generalized replacing Schauder
basis with Markushevich basis. The argument presented here is not
directly applicable, since, for example, we have made use of the canonical
projections on the basis and their uniform boundedness.

\medskip{}

\textbf{Acknowledgments.} The authors wish to thank the referee for a
careful reading of our manuscript and for pointing out to us the above question.

\end{document}